\documentclass[12pt]{amsart}
\usepackage{main}
\usepackage{graphicx}
\usepackage{float}
\usepackage{hyperref}
\usepackage{epstopdf}

\def\hat{\widehat}
\def\wtilde{\widetilde}
\def\D{\mathcal D}

\def\bra{\langle}
\def\cet{\rangle}
\newtheorem*{theorem*}{Theorem}

\def\FM{\mathcal A}

\def \p {\partial}

\def \ba {\begin {eqnarray*} }	
\def \ea {\end {eqnarray*} }
\def \beq {\begin {eqnarray}}
\def \eeq {\end {eqnarray}}

\newcommand{\sijoitus}[2]%
{\operatornamewithlimits{\Bigl/}_{\!\!\!#1}^{\,#2}}

\title[Regularization for wave equation]{Regularization strategy for inverse problem for 1+1 dimensional wave equation}
\author{Jussi Korpela, Matti Lassas and Lauri Oksanen}
\date{August 31, 2015}
\begin{document}
\maketitle
\begin{abstract}
An inverse boundary value problem for a 1+1 dimensional wave equation
with wave speed $c(x)$ is considered. We give a regularisation
strategy for inverting the map $\mathcal A:c\mapsto \Lambda,$ where
$\Lambda$ is the hyperbolic Neumann-to-Dirichlet map corresponding to the wave speed $c$.
More precisely, we consider the case when we are given a perturbation of the Neumann-to-Dirichlet map
$\tilde \Lambda=\Lambda +\mathcal E $, where $\mathcal E$ corresponds to the measurement errors, and reconstruct an approximate wave speed $\tilde c$.
We emphasize that $\tilde \Lambda$ may not not be in the range of the map $\mathcal A$.
We show that the
reconstructed wave speed $\tilde c$ satisfies $\| \tilde c-c\|_{L^\infty}<C \|E\|^{1/18}$.
Our regularization strategy is based on a new formula to compute $c$ from $\Lambda$.
\smallskip

\end{abstract}

\section{Introduction}
We consider an inverse boundary value problem for the wave equation $$(\p_t^2 - c(x)^2 \p_x^2) u(t,x) = 0.$$
We introduce a regularization strategy to recover the sound speed $c(x)$ by using the knowledge of perturbed Neumann-to-Dirichlet map $\wtilde\Lambda$.
Our approach is based on the Boundary Control method \cite{Belishev1987, Belishev1992, Tataru1995}.

A variant of the Boundary Control method, called the iterative time-reversal control method, was introduced in \cite{Bingham2008}.
The method was later modified in \cite{Dahl2009}
to focus the energy of a wave at a fixed time 
and in \cite{Lauri2013} 
to solve an inverse obstacle problem for the wave equation. 
Here we introduce yet another modification of the iterative time-reversal control method
that is tailored for the 1+1 dimensional wave equation.

Classical regularization theory is explained in \cite{Engl1996}. Iterative regularization
of both linear and nonlinear inverse problems and convergence rates are discussed
in Hilbert space setting in \cite{Bissantz2004, Hanke1997, Hohage2008, LuPeRa2007, Mathe2008} and in Banach space setting in \cite{Hofmann2007, Kaltenbacher2006, Kaltenbacher2008, Kirsch1996, Ramlau2008, Ramlau2006, Resmerita2005}. Our new results give a direct regularization method for the
nonlinear inverse problem for the wave equation. The result contains an
explicit (but not necessarily optimal) convergence rate.

By direct methods for non-linear problems we mean explicit construction of a non-linear map solving the problem
without resorting to a local optimisation method. In our case the map is given by (\ref{Rstategia}).
The advantage of direct approaches
is that they do not suffer from the possibility that the algorithm converges to a local minimum.
In particular, they do not require a priori knowledge
that the solution is in a small neighbourhood of a given function.
There are currently only few regularized direct methods for non-linear inverse problems.
An example is a regularisation algorithm for
the inverse problem for the conductivity equation in \cite{KLMS2009}.
 Also, a direct regularized inversion for blind
deconvolution is presented in \cite{Justen2006}.
\subsection{Statement of the results}
 Let $C_0,C_1,L,m>0$ and define the space of velocity functions
\begin{align}
\label{nopeudet}
\mathcal{D}(\mathcal A):=&\{c\in L^\infty (M);
C_0\le c(x)\le C_1 ,\\\nonumber &\quad\norm{c}_{C^2(M)}\le m,\, c-1 \in C_0^2(0,L) \},
\end{align}
where we denote by $M$ the half axis $[0,\infty)$.
Let 
\begin{align}
\label{aikapoika}
T> \frac{L}{C_0}.
\end{align}
For $c\in\mathcal{D}(\mathcal A)$ and 
$f\in L^2(0,2T)$, the boundary value problem 
\begin{align}
\label{dartwader}
&(\p_t^2 - c(x)^2 \p_x^2) u(t,x) = 0 \quad \text{in $(0,2T) \times M$},
\\\nonumber
&\p_x u(t,0) = f(t),
\\\nonumber
&u|_{t = 0} = \p_t u|_{t=0} = 0,
\end{align}
has a unique solution $u^f\in H^1((0,2T) \times M)$. Using this solution we define the Neumann-to-Dirichlet operator
\begin{align}
\label{lam}
 \Lambda : L^{2}(0,2T) \to L^{2}(0,2T),\quad
 \Lambda f=u^f|_{x=0}.
\end{align}
We define for a Banach space E 
\begin{align*}
\mathcal{L} (E):=\{A:E \to E;
A\text{ is linear and continuous}\}.
\end{align*}
Let $X=L^\infty (M)$ and $Y=\mathcal{L} (L^2(0,2T))$.
We define the direct map 
\begin{align}
\label{apina}
\mathcal A:\mathcal{D}(\mathcal A)\subset X\to Y,\quad 
 \mathcal A(c)= \Lambda. 
\end{align}
We show in Appendix A, Theorem \ref{tomoffinland}, that the maps (\ref{lam}) and (\ref{apina}) are continuous. Here the range $Ran (\mathcal{A})=\mathcal A(\mathcal{D}(\mathcal A))$ and the domain $\mathcal{D}(\mathcal A)$ are equipped with the topologies of $Y$ and $X$, respectively.

We consider the inverse problem to recover the velocity function $c$ by using the boundary measurements $\Lambda$. It is well-known that $\mathcal A$ is invertible. Let us record the following:
\begin{theorem}The inverse map
\label{Amiinus}
\begin{align*}
\mathcal A^{-1}:Ran (\mathcal{A})\subset Y\to \mathcal{D}(\mathcal A)\subset X,\quad
\mathcal A^{-1}(\Lambda)= c, 
 \end{align*}
 is continuous.
 \end{theorem}
For the convenience of the reader we give a proof of Theorem \ref{Amiinus} in Section 2, where we also give a new formula to compute $c$ from $\Lambda$. Our main result concerns perturbations of the Neumann-to-Dirichlet operator of the form  
\begin{align}
\label{virhemitta}
&\wtilde{\Lambda}=\Lambda+\mathcal{E}, 
\end{align}
where $\mathcal{E}\in Y$
models the measurement error. 
We assume that $\norm{\mathcal{E}}_Y \le \epsilon,$
where $\epsilon>0$ is known. In this situation we can not use the map $\mathcal A^{-1}$ to calculate function $c$ since $\wtilde{\Lambda}$ may not be in the range $\mathcal R (\mathcal{A})$. We recall the definition of a regularization strategy, see e.g. \cite{Engl1996} and \cite{Kirsch1996}. 
\begin{definition}
Let X,Y be Banach spaces and $\mathcal{D}(\mathcal A)\subset X$. Let $\mathcal A:\mathcal{D}(\mathcal A)\to Y$ be continuous mapping. Let $\alpha_0\in (0,\infty]$. A family of continuous maps $\mathcal R_{\alpha}:Y\to X$ parametrized by $0<\alpha<\alpha_0$ is called a regularization strategy if
\begin{align*}
\lim_{ \alpha\to 0} \mathcal R_{\alpha}(\mathcal A(c))=c
\end{align*}
for every $c\in\mathcal{D}(\mathcal A)$.
A regularization strategy is called admissible, if the parameter $\alpha$ is chosen as a function of the noise level $\epsilon$ so that $\lim_{\epsilon \to 0}$ $\alpha(\epsilon) =0$ and for every $c\in\mathcal{D}(\mathcal A)$ 
\begin{align*}
\lim_{\epsilon \to 0}\sup \Big\{\norm{\mathcal R_{\alpha(\epsilon)}\wtilde{\Lambda}-c}_X:
\wtilde{\Lambda}\in Y,\,\norm{\wtilde{\Lambda}-\mathcal A(c)}_Y \le \epsilon \Big\}= 0.
\end{align*}
\end{definition}

  Figure \ref{fig:nonlinearregul} gives us a schematic illustration of regularization.
 \begin{figure}[H]
 \begin{picture}(300,145)
 \put(0,0){\includegraphics[width=12cm]{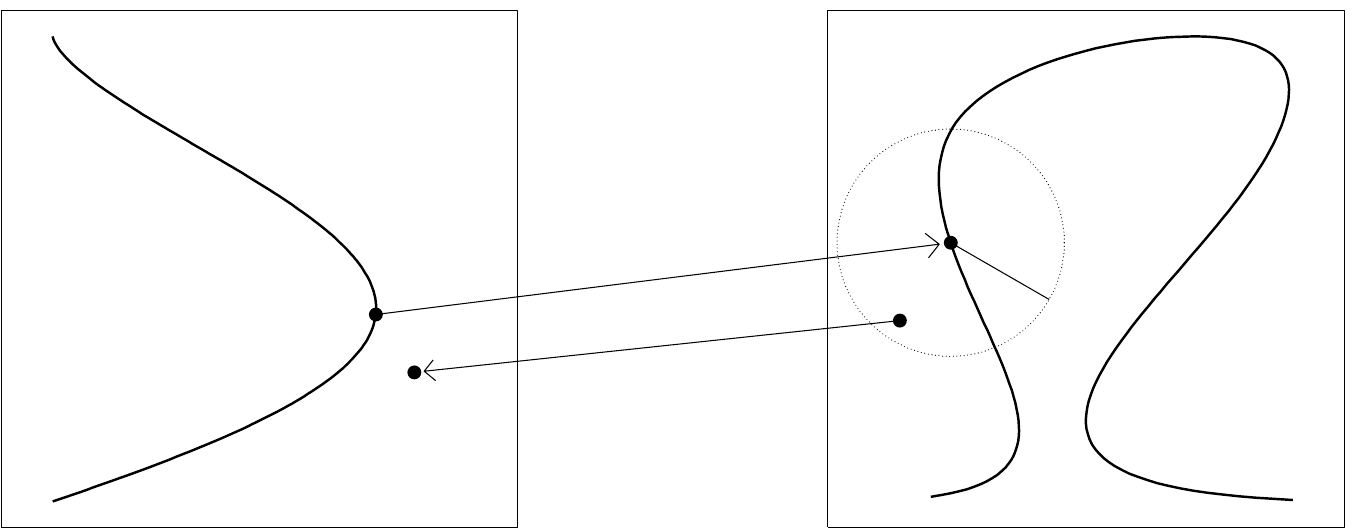}}
 \put(0,135){\small Model space}
 \put(207,135){\small \small Data space}
 \put(115,117){\small $X$}
 \put(215,117){\small $Y$}
 \put(85,50){\small $c$}
 \put(160,68){\small $\FM$}
 \put(260,60){\small $\epsilon$}
 \put(160,33){\small $\mathcal R_{\alpha(\epsilon)}$}
 \put(240,75){\small $\FM(c)$}
 \put(290,18){\small $\FM(\D(\FM))$}
 \put(10,18){\small $\D(\FM)$}
 \put(230,50){\small $\wtilde \Lambda$}
 \put(88,28){\small $\mathcal R_{\alpha(\epsilon)}(\wtilde \Lambda)$}
 \end{picture}
 \caption{\label{fig:nonlinearregul}\index{regularized inversion}The idea of regularization is to construct a family $\mathcal R_{\alpha(\epsilon)}$ of continuous maps from the data space $Y$ to the model space $X$ in such a way that $c$ can be approximately recovered from noisy data $\wtilde \Lambda$. For a smaller noise level $\epsilon$ the approximation $\mathcal R_{\alpha(\epsilon)}(\wtilde \Lambda)$ is closer to $c$. More details and a similar figure in a more general setting can be found in \cite[Fig. 11.5]{mueller2012linear}.}
 \end{figure}
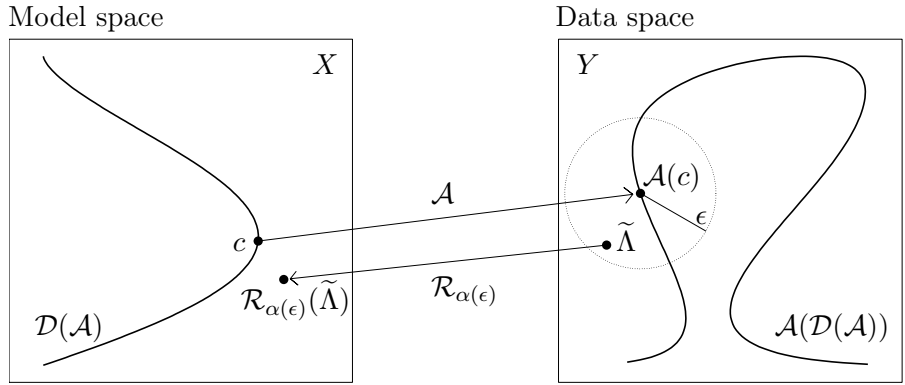
The main result is presented next and says that we have an admissible regularization strategy inverting $\mathcal A$. 
\begin{theorem}
\label{kaiken_teoria}
There exists an admissible regularization strategy  $\mathcal R_{\alpha}$ with the choice of parameter 
\begin{align*}
\alpha(\epsilon)=2^{\frac{13}{9}}T^{\frac{4}{9}}\epsilon^{\frac{4}{9}}
\end{align*}
that satisfies the following: For every $c\in\mathcal{D}(\mathcal A)$
there is $\epsilon_0$, $C>0$ such that
\begin{align*} 
\sup\Big\{\norm{\mathcal R_{\alpha(\epsilon)}\wtilde{\Lambda}-c}_X:
\wtilde{\Lambda}\in Y,\,\norm{\mathcal A(c)-\wtilde{\Lambda}}_Y \le \epsilon \Big\}\le
C\epsilon ^{\frac{1}{18}},
\end{align*}
for all $\epsilon \in (0,\epsilon_0).$ 
\end{theorem}
We will give explicit choices of $\mathcal R_{\alpha}$ and $\epsilon_0$, see (\ref{epsilonnolla}) and (\ref{Rstategia}) below.
\subsection{Previous literature}
From the point of view of uniqueness questions, the inverse problem
for the 1+1 dimensional wave equation is equivalent with the one
dimensional inverse boundary spectral problem. The latter problem was
thoroughly studied in 1950s  \cite{GeLe1951, Krein1951, Marchenko1950} and we refer to \cite[pp.\ 65-67]{Kabanikhin2005} for a historical overview. In
1960s Blagove{\v{s}}{\v{c}}enski{\u\i} \cite{Blagovevsvcenskiui1969, Blagovevsvcenskiui1971} developed an
approach to solve the inverse problem for the 1+1 dimensional wave
equation without reducing the problem to the inverse boundary spectral
problem. This and later dynamical methods have the advantage over
spectral methods that they require data only on a finite time
interval.

The method in the present paper is a variant of the Boundary Control
 method  that was pioneered by M. Belishev \cite{Belishev1987}
and developed by  M. Belishev and Y. Kurylev \cite{BeKu1987,Belishev1992} in late 80s and early 90s.
Of crucial importance for the method was the result of D.\ Tataru
\cite{Tataru1995} concerning
a Holmgren-type uniqueness theorem for non-analytic coefficients.
The Boundary Control method for multidimensional inverse problems has been summarized in \cite{Belishev1997, KKL2001},
and considered for 1+1 dimensional scalar problems in \cite{BeKa1989, BeRFi1994} and for  multidimensional scalar problems
in \cite{Katchalov1998, KKLM2004, Kury1995, LaOk2014, LaOk2014b}
For systems it has been considered in \cite{Kurylev2009, Kurylev2006}.
Stability results for the method have been considered in \cite{Anderson2004}
and \cite{Katsuda2007}.

The inverse problem for the wave equation can be solved also by using complex geometrical optics solutions. These solutions were developed in the context of elliptic
inverse boundary value problems \cite{Sylvester1987}, and in \cite{Nachman1988} they were employed to solve an inverse
boundary spectral problem. Local stability results can be proven using
(real) geometrical optics solutions \cite{Bellassoued2011, StUh1998, Stefanov2011}, and in \cite{Liu2012} a local stability result was
proved by using ideas from the Boundary Control method together with complex
geometrical optics solutions. Finally we mention the important method
based on Carleman estimates \cite{Bukhgeuim1981} that can be used to show stability
results when the initial data for the wave equation is non-vanishing.
\section{Modification of the iterative time-reversal control method}
In this section we prove Theorem \ref{Amiinus} in such a way that we can  utilize the proof to construct a regularization stratezy as in Theorem \ref{kaiken_teoria}. 
Let $\Lambda$ be as defined in (\ref{lam}). Let $r\in [0,T]$. We define linear operators in $Y$ by
\begin{align}
\label{operaattorit}
&J f(t) = \frac{1}{2} \int_0^{2 T} 1_\triangle (t,s)f(s) ds,
\\\nonumber& R f(t) = f(2 T - t), \quad
 K = J \Lambda - R \Lambda R J,
\\\nonumber& B f (t) = 1_{(0,T)}(t) \int_t^T f(s) ds,\quad
 P_r f(t) = 1_{(0,r)}(t)f(t),
\end{align}
where 
$$
1_{\triangle} (t,s) = \begin{cases}
1, & t + s \le 2 T$ \text{and} $s > t > 0,
\\
0, & \text{otherwise}
\end{cases} 
$$
and
$$
1_{(0,r)}(t)= \begin{cases}
1, & t\in(0,r),
\\
0, & \text{otherwise}.
\end{cases} 
$$
Let $f\in L^2(0,2T)$. Using solution $u^f\in H^1((0,2T) \times M)$ of (\ref{dartwader}) we define 
\begin{align}
\label{nasu}
 U_{T} : L^{2}(0,2T)\mapsto H^1(M),\quad 
U_{T}f=u^f(T).
\end{align}
Let us denote $dV = c^{-2} dx$ and recall the Blagovestchenskii identities 
\begin{align}
\label{Blago_harmonic}
&\bra u^f(T), 1\cet _{L^2(M;dV)} 
= \bra f, B 1\cet_{L^2(0, 2T)},
\\\nonumber& \bra u^f(T), u^h(T)\cet_{L^2(M;dV)} 
= \bra f, K h\cet_{L^2(0, 2T)}.
\end{align}
The identities (\ref{Blago_harmonic}) originate from the work by Blagovestchenskii \cite{Blagovevsvcenskiui1966} and their proofs can be found e.g. in \cite{Bingham2008}.
We define the domain of influence
\begin{align}
\label{Mr}
M(r) = \{ x \in M; d(x, 0) \le r \},
 \end{align}
where $d(x, 0) = \int_0^x \frac{1}{c(t)} dt$.
We use the following result that is closely related to \cite{Bingham2008}
\begin{theorem}
\label{thm_minimization}
Let $r \in [0, T]$ and $\alpha > 0$. Let $K,B$, and $P_r$ be as defined in (\ref{operaattorit}). Let us define
\begin{equation}
\label{muumipappa}
S_r=\{f \in L^2(0,2T)\,:\,\supp(f) \subset [T - r, T]\}.
\end{equation}
Then the regularized minimization problem 
\begin{equation*}
\min_{f \in S_r} \ll (\bra f, K f\cet _{L^2(0, 2T)} - 2\bra f, B1\cet _{L^2(0, 2T)} + \alpha \norm{f}_{L^2(0, 2T)}^2 \rr),
\end{equation*}
has unique minimizer 
\begin{equation}
\label{minimizer}
 f_{\alpha,r} = (P_rKP_r + \alpha)^{-1}P_r B1
\end{equation}
and the map $r\mapsto f_{\alpha,r}$
is continuous $[0,T]\to L^2(0, 2T)$.
Moreover $u^{f_{\alpha,r}}(T)$ converges to the indicator function of the 
domain of influence,
\begin{align}
&\lim_{\alpha\to 0}\norm{u^{f_{\alpha,r}}(T) - 1_{M(r)}}_{L^2(M;dV)} =0.
\end{align}
\end{theorem}
For the convenience of the reader we give a proof.
\begin{proof}[Proof of Theorem \ref{thm_minimization}.]
Let $\alpha>0$ and let $f\in S_r$. We define the energy function 
\begin{equation}
\label{appe1}
E(f) := \bra f, Kf\cet _{L^2(0,2T)}  - 2 \bra f, B1\cet _{L^2(0,2T)} + \alpha \norm{f}^2_{L^2(0,2T)}.
\end{equation}
The finite speed of wave propagation
implies $\supp(u^f(T)) \subset M(r)$. Using (\ref{Blago_harmonic}) we can write 
\begin{equation}
\label{eq:energy_functional}
E(f) = \norm{u^{f}(T) - 1_{M(r)}}_{L^2(M; dV)}^2 - \norm{1_{M(r)}}_{L^2(M; dV)}^2 + \alpha \norm{f}^2_{L^2(0,2T))}.
\end{equation}
Let $(f_j)_{j=1}^\infty \subset S_r$ be such that 
\begin{equation*}
\lim_{j \to \infty} E(f_j) = \inf_{f \in S_r} E(f).
\end{equation*}
Then 
\begin{equation*}
\alpha \norm{f_j}_{L^2(0,2T))} \le E(f_j) + \norm{1_{M(r)}}_{L^2(M; dV)}^2,
\end{equation*}
and we see that $(f_j)_{j=1}^\infty$ is bounded in $S_r$.
As $S_r$ is a Hilbert space,
there is a subsequence of $(f_j)_{j=1}^\infty$ converging weakly in $S_r$.
Let us denote the limit by $f_\infty \in S_r$ and the subsequence still by $(f_j)_{j=1}^\infty$. 

By Theorem \ref{tomoffinland} in Appendix A below, the map $U_T:f \mapsto u^f(T)$ is bounded
\begin{equation*}
U_T:L^2( 0, 2T) \to H^{1}(M).
\end{equation*}
The embedding $I:H^{1}(M)\hookrightarrow L^2(M)$ is compact
and thus $U_T:f \mapsto u^f(T)$ is a compact operator 
\begin{equation*}
U_T:L^2(0, 2T) \to L^2(M).
\end{equation*}
Hence we have a subsequence $(f_j)_{j=1}^\infty$ for which $u^{f_j}(T) \to u^{f_\infty}(T)$ in $L^2(M)$ as $j \to \infty$. Moreover, the weak convergence implies 
\begin{equation*}
\norm{f_\infty}_{L^2(0,2T))} \le \liminf_{j \to \infty} \norm{f_j}_{L^2(0,2T)}.
\end{equation*}
Hence
\begin{align*}
&E(f_\infty) = \lim_{j \to \infty} \norm{u^{f_j}(T) - 1_{M(r)}}_{L^2(M; dV)}^2 - \norm{1_{M(r)}}_{L^2(M; dV)}^2 + \alpha \norm{f_\infty}^2_{L^2(0,2T)}
\\\nonumber&\le \lim_{j \to \infty} \norm{u^{f_j}(T) - 1_{M(r)}}_{L^2(M; dV)}^2 - \norm{1_{M(r)}}_{L^2(M; dV)}^2 + \alpha 
\liminf_{j \to \infty} \norm{f_j}^2_{L^2(0,2T)}
\\\nonumber&= \liminf_{j \to \infty} E(f_j) = \inf_{f \in S_r} E(f),
\end{align*}
and thus $f_\infty \in S_r$ is a minimizer for (\ref{appe1}). We denote by $D_h$ the Fr\'echet derivative to direction $h$. If
\begin{equation*}
0 = D_h E(f) = 2  \bra h, P_rKP_r f\cet_{L^2(0, 2T)} - 2 \bra h, P_rB 1\cet_{L^2(0, 2T)} + 2 \alpha \bra h, f\cet_{L^2(0, 2T)},
\end{equation*}
for all $h\in L^2(0,2T)$, then
\begin{equation*}
(P_rKP_r + \alpha) f = P_r B 1.
\end{equation*}
Using (\ref{Blago_harmonic}) we have
\begin{align*}
&\bra (P_rKP_r+\alpha)f,f \cet_{L^2(0, 2T)}=\bra u^{P_rf}(T), u^{P_rf}(T)\cet_{L^2(M;dV)}
+\bra \alpha f,f\cet _{L^2(0, 2T)}.
\end{align*} 
Operator $P_rKP_r+\alpha$ is coersive when $\alpha > 0$. The Lax-Milgram Theorem implies that it is invertible, and we have an expression for minimizer
\begin{equation*}
 f_{\alpha,r}:=f_\infty = (P_rKP_r + \alpha)^{-1}P_r B 1.
\end{equation*}
According to \cite{Tataru1995}, see also \cite{Katchalov2001}, we know that
\begin{equation*}
\{ u^f(T) \in L^2(M(r));\ f \in S_r \}
\end{equation*}
is dense in $L^2(M(r))$. 
Let $\epsilon > 0$ and fix $f_\epsilon \in S_r$ such that
\begin{equation}
\label{rudiwölleri}
\norm{u^{f_\epsilon}(T) - 1_{M(r)}}_{L^2(M; dV)}^2 \le \epsilon.
\end{equation}
Using (\ref{eq:energy_functional}) we have
\begin{align*}
\norm{u^{f_{\alpha,r}}(T) - 1_{M(r)}}_{L^2(M; dV)}^2 \le  
E(f_{\alpha,r})+\norm{1_{M(r)}}_{L^2(M; dV)}^2.
\end{align*}
Because $E(f_{\alpha,r}) \le E(f_\epsilon)$ 
we have
\begin{align*}
\norm{u^{f_{\alpha,r}}(T) - 1_{M(r)}}_{L^2(M; dV)}^2 & \le  
 \norm{u^{f_\epsilon}(T) - 1_{M(r)}}_{L^2(M; dV)}^2 + \alpha \norm{f_\epsilon}^2.
\\\nonumber& \le \epsilon + \alpha \norm{f_\epsilon}^2.
\end{align*}
Using (\ref{rudiwölleri}) we may choose first small $\epsilon > 0$ and then small $\alpha > 0$ to 
see that $u^{f_\alpha}(T)$ tends to $1_{M(r)}$ in $L^2(M)$ as $\alpha\to 0$. 
\end{proof}

See Figure \ref{fig:sika} for a visualization of $M(r)$.
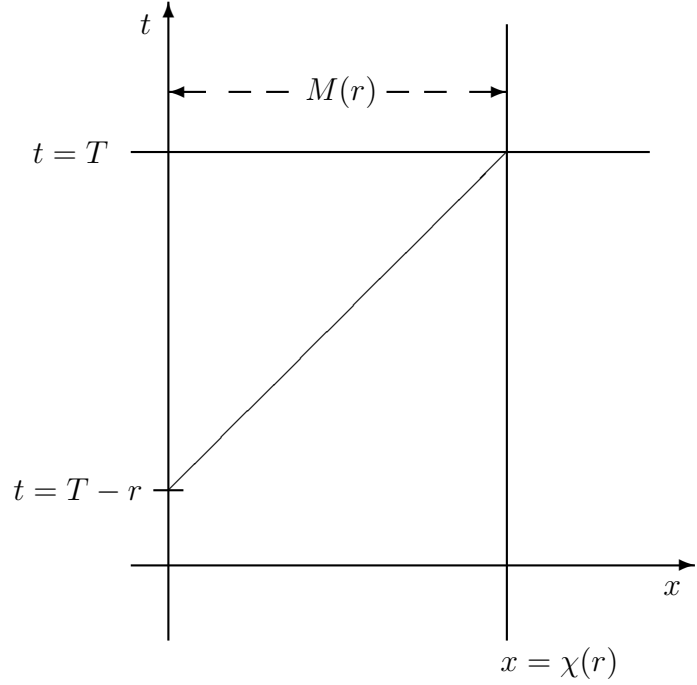
\begin{figure}[H]
\setlength{\unitlength}{1mm}
\begin{picture}(120,90)(-20,0)
\thicklines
  \put(15,15){\vector(1,0){75}}
   \put(20,5){\vector(0,1){85}}
   \put(65,5){\line(0,1){82}}
   \put(15,70){\line(1,0){69}}
   \put(18,25){\line(1,0){4}}
   
  \put(87,12){\makebox(0,0){$x$}}
   \put(17,87){\makebox(0,0){$t$}}
  \put(7,70){\makebox(0,0){$t=T$}}
  \put(8,25){\makebox(0,0){$t=T-r$}}
  \put(72,2){\makebox(0,0){${x=\chi(r)}$}}
  \put(43,78){\makebox(0,0){$M(r)$}}
  \multiput(28,78)(5,0){2}{\line(1,0){3}}
  \multiput(49,78)(5,0){2}{\line(1,0){3}}
  \put(25,78){\vector(-1,0){5}}
   \put(60,78){\vector(1,0){5}}
 \thinlines
\put(20,25){\line(1,1){45}}
\end{picture}
\caption{\label{fig:sika}When sending information from an interval of length $r$, that is when $\supp(f) \subset[T-r,T]$, the solution $u^f(x,T)$ is supported in the domain of influence ${ M(r)}$. }
\end{figure}
We define the travel time coordinates for $x \in M$ by
\begin{align*}
\tau:[0,\infty)\to[0,\infty), \quad \tau(x) = d(x, 0).
\end{align*}
The function $\tau$ is strictly increasing and we denote its inverse by
\begin{align*}
\chi=\tau^{-1}:[0,\infty)\to[0,\infty).
\end{align*}
We have
\begin{align}
\label{kiipilkku}
\chi(0) = 0, \quad \chi'(t) = \frac{1}{\tau'(\chi(t))} = c(\chi(t)).
\end{align}
 Thus denoting $v(t) = c(\chi(t))$ and use $V(r)$ to denote the volume of $M(r)$ with respect to the measure $dV$ we have  
\begin{align}
\label{Vee}
V(r) 
=\norm{1_{M(r)}}_{L^2(M;dV)}^2 = \int_0^{\chi(r)} \frac{dx}{c(x)^2}
= \int_0^r \frac{\chi'(t) dt}{v(t)^2}
= \int_0^r \frac{dt}{v(t)}.
\end{align}
Note that $M(r)=[0,\chi(r)]$. In particular, $V(r)$ determines the wave speed in the travel time coordinates,
\begin{align}
\label{pikkuvee}
v(r) = \frac{1}{\p_r V(r)},
\end{align}
and also in the original coordinates since
\begin{align}
\label{cee}
c(x) = v(\chi^{-1}(x)), 
\quad \chi(t) = \int_0^t v(t^\prime) dt^\prime.
\end{align}
Using Theorem \ref{thm_minimization} and (\ref{Blago_harmonic}) we have a method to compute 
the volumes of the domains of influence
\begin{align}
\label{possu}
V(r) = \norm{1_{M(r)}}_{L^2(M;dV)}^2 = \lim_{\alpha \to 0} \bra f_{\alpha,r}, B1\cet _{L^2(0, 2T)},
\end{align}
where $r\in [0,T]$. 
We are ready to prove Theorem \ref{Amiinus}.
\begin{proof}[Proof of Theorem \ref{Amiinus}.] For a given measurement $\Lambda$, Theorem \ref{thm_minimization} and equations (\ref{pikkuvee}), (\ref{cee}), (\ref{possu}) give us a way to calculate for all $x\in (0,L)$ the value of the velocity function
\begin{align*}
c(x)= v(\chi^{-1}(x))=\mathcal A^{-1}(\Lambda)(x).
 \end{align*}
As we assumed that outside of the interval $(0,L)$ the function c is identically one, the proof for the existence of inverse map $\mathcal A^{-1}$ is complete. 
Since the direct map $\mathcal A$ is continuous and $\mathcal{D}(\mathcal A)$ is relatively compact in $X$, we see that $\mathcal A^{-1}$ is a continuous map.
 \end{proof}

\section{Stability of regularized problem}
In this section we prove Theorem \ref{kaiken_teoria}. We will construct the operator $\mathcal R_{\alpha(\epsilon)}$ as a composition of several operators. The construction is motivated by the proof of Theorem \ref{Amiinus}. We define for a Banach space $E$
\begin{align*}
\mathcal{K} (E)=\{A\in\mathcal{L} (E) ;
A\text{ is compact}\}.
\end{align*}
Let $J$, $R$ be as defined in (\ref{operaattorit}). Using (\ref{operaattorit}) we see that $J\in\mathcal{K} (L^2(0,2T))$. We define 
\begin{align}
\label{Hoo}
&\boldsymbol{K} :Y \to \mathcal{K}(L^2(0,2T)), 
\quad \boldsymbol{K}\wtilde\Lambda=    
J\wtilde\Lambda - R\wtilde\Lambda RJ.
\\\nonumber& \boldsymbol{H} :Y \mapsto C([0,T],Y) , 
\quad \boldsymbol{H}\wtilde\Lambda=r\mapsto    
P_r(\boldsymbol{K}\wtilde\Lambda)P_r.
\end{align}
\begin{proposition}
\label{K}
We have $\norm{\boldsymbol{H}}_{Y \to C([0,T],Y)}\le 2T$. 
\end{proposition}
\begin{proof}Let $r\in [0,T]$. We have estimates $\norm{P_r}_{Y}\le 1$, $\norm{R}_{Y}\le 1$, $\norm{J}_{Y}\le T$, and
\begin{align*}
\norm{\boldsymbol{H}\wtilde{\Lambda}(r)}_{\mathcal{L}(L^2(0,2T))} \le2\norm{J}_{\mathcal{L}(L^2(0,2T))}\norm{\wtilde{\Lambda}}_{\mathcal{L}(L^2(0,2T))}\le
2T \norm{\wtilde{\Lambda}}_{\mathcal{L}(L^2(0,2T))}.
\end{align*}
Thus
\begin{align*}
\norm{\boldsymbol{H}}_{Y\to L^{\infty} ([0,T],Y)} \le
2T .
\end{align*}
It remains to show that $r\mapsto \boldsymbol{H}\wtilde\Lambda(r)$ is continuous. Let us denote $\wtilde K=\boldsymbol{K}\wtilde\Lambda$. 
Let $r,s\in[0,T]$. We use the singular value decomposition for the compact operator $\wtilde K$. There are orthonormal bases $\{\phi_n\}_{n=1}^\infty\in L^2(0,2T)$ and $\{\psi_n\}_{n=1}^\infty\in L^2(0,2T)$ such that
\begin{align}
\label{nakustelijat}
\wtilde Kf=\sum_{n=1}^{\infty}\mu_n\bra f,\phi_n\cet_{L^2(0,2T)} \psi_n,
\end{align}
for all $f\in L^2(0,2T)$, where $\mu_n\in\mathbb R$ are the singular values of $\wtilde K$.
We define the family $\{\wtilde K^m\}_{m=1}^\infty$ of finite rank operators by the formula 
\begin{align}
\label{nakustelijat2}
\wtilde K^mf=\sum_{n=1}^{m}\mu_n\bra f,\phi_n\cet_{L^2(0,2T)} \psi_n.
\end{align}
Then
\begin{align}
\label{seriffi}
&\norm{P_r\wtilde{K}P_rf-P_s\wtilde{K}P_sf}_{L^2(0,2T)}
\\\nonumber& \le \norm{P_r\wtilde{K}P_rf-P_r\wtilde{K}^mP_rf}_{L^2(0,2T)}+
\norm{P_r\wtilde{K}^mP_rf-P_s\wtilde{K}^mP_rf}_{L^2(0,2T)}+
\\\nonumber&  \norm{P_s\wtilde{K}^mP_rf-P_s\wtilde{K}^mP_sf}_{L^2(0,2T)}+
\norm{P_s\wtilde{K}^mP_sf-P_s\wtilde{K}P_sf}_{L^2(0,2T)}.
\end{align}
Let $\epsilon>0$ and let $\norm{f}_{L^2(0,2T)}\le1$. By choosing $m$ large enough we have 
\begin{align*}
\norm{P_r\wtilde{K}P_rf-P_r\wtilde{K}^mP_rf}_{L^2(0,2T)}+
\norm{P_s\wtilde{K}^mP_sf-P_s\wtilde{K}P_sf}_{L^2(0,2T)}\le \frac{\epsilon}{2}.
\end{align*}
Applying projections to (\ref{nakustelijat2}) we see that
\begin{align*}
P_s\wtilde K^mP_rf=\sum_{n=1}^{m}\mu_n\bra f,P_r\phi_n\cet P_s\psi_n.
\end{align*}
For the second term in the sum (\ref{seriffi}) we have an estimate
\begin{align*}
&\norm{P_r\wtilde{K}^mP_rf-P_s\wtilde{K}^mP_rf}_{L^2(0,2T)}=
\norm{\sum_{n=1}^{m}\mu_n\bra f,P_r\phi_n\cet (P_r-P_s)\psi_n}_{L^2(0,2T)}
\\\nonumber& \le\sum_{n=1}^{m}|\mu_n|\norm{(P_r-P_s)\phi_n}_{L^2(0,2T)}
\le C(m)|r-s|^{\frac{1}{2}}.
\end{align*}
For the third term in the sum we have an analogous estimate
\begin{align*}
\norm{P_s\wtilde{K}^mP_rf-P_s\wtilde{K}^mP_sf}_{L^2(0,2T)}=
\le C(m)|r-s|^{\frac{1}{2}}.
\end{align*}
Putting these estimates together and choosing $|r-s|\le \delta(\epsilon)= \frac{\epsilon^2}{4C(m)^2}$,
we see that
\begin{align*}
&\norm{P_r\wtilde{K}P_r-P_s\wtilde{K}P_s}_{Y}
\le \epsilon.
\qedhere
\end{align*}
\end{proof}
Let us define 
\begin{align}
\label{hallaaho}
M_1=\sup\{\norm{\mathcal A(c)}_{\mathcal{L} (L^2(0,2T))};c\in\mathcal{D}(\mathcal A)\}.
\end{align}
Using the continuity of $\mathcal A$, see Theorem \ref{tomoffinland} below, we see that $M_1 <\infty$.
We define $M_2=2TM_1$.
Let $c\in \mathcal{D}(\mathcal A)$ and denote $\Lambda=\mathcal A(c)$. We use again the notations $H=\boldsymbol{H}\Lambda$, $\wtilde H=\boldsymbol{H}\wtilde\Lambda$ and $\wtilde H_r=\boldsymbol{H}\wtilde\Lambda(r)$.  Using Proposition \ref{K} we have
\begin{align}
\norm{H}_{C([0,T],Y)}\le M_2.
\end{align}
We define $M_3=M_2+3$ and a family $\{\Psi^Z_{\alpha}\}_{\alpha\in(0,2]}\in C(\mathbb{R})$ by
\begin{displaymath}
\Psi^Z_{\alpha}(s)= \left\{ \begin{array}{cl}
0, & \textrm{if\quad $s>M_3-\frac{\alpha}{4}$},\\
-\frac{4}{\alpha}s+\frac{4M_3}{\alpha}-1, & \textrm{if\quad $s\in(M_3-\frac{\alpha}{2},M_3-\frac{\alpha}{4}]$},\\
1, & \textrm{if\quad   $s\le M_3-\frac{\alpha}{2}$}.\\
\end{array} \right.
\end{displaymath}
For $\alpha\in (0,2]$ we define 
\begin{align}
\label{Zalfa}
 \boldsymbol{Z}_{\alpha}:C([0,T],Y) \to C([0,T],Y) ,
\end{align}
\begin{displaymath}
\boldsymbol{Z}_{\alpha} \big(\wtilde H\big )= 
r\mapsto\Psi^Z_{\alpha}\Big(\norm{M_3-(\wtilde H+\alpha)}_{C([0,T],Y)}\Big)\Big(\wtilde H_r+\alpha\Big)^{-1}.
\end{displaymath}
Let $E$ be a Banach space and let $H\in E$. Let $\epsilon >0$. We denote 
\begin{align}
\label{pallo}
 \mathcal{B}_E(H,\epsilon):=\{\wtilde H\in E : \norm{H-\wtilde{H}}_E < \epsilon \}.
\end{align}
\begin{proposition}
\label{Z2}
Let $\epsilon\in (0,1)$ and let $p\in (0,\frac{1}{2})$. Let $\alpha=2\epsilon^p$ and let $\norm{H}_{C([0,T],Y)}\le M_2$. Let $H_r \in Y$ be positive semidefinite. Let us assume that $\wtilde H\in \mathcal{B}_{C([0,T],Y)} (H,\epsilon)$. Then
\begin{align*}
\norm{\boldsymbol{Z}_{\alpha} \big( H\big ) - \boldsymbol{Z}_{\alpha} \big(\wtilde H\big )}_{C([0,T],Y)} \le
2^{-1}\epsilon^{1-2p}.
\end{align*}  
\end{proposition}
\begin{proof}
Using (\ref{Zalfa}) we see that if $\Psi_{\alpha}\Big(\norm{M_3-(\wtilde H+\alpha)}_{C([0,T],Y)}\Big)\neq 0$, then
\begin{align*}
\norm{M_3-(\wtilde H+\alpha)}_{C([0,T],Y)}\le M_3-\frac{\alpha}{4}<M_3
\end{align*}
and $\Big(\wtilde H_r+\alpha\Big)^{-1}$ is defined by the formula
\begin{align*}
\Big(\wtilde H_r +\alpha\Big)^{-1}=\frac{1}{M_3}\Big(I-\frac{M_3-(\wtilde H_r +\alpha)}{M_3}\Big)^{-1}
=\frac{1}{M_3}\sum_{l=1}^\infty\Big(\frac{M_3-(\wtilde H_r +\alpha)}
{M_3}\Big)^l.
\end{align*}
This gives that $\boldsymbol{Z}_{\alpha} \big(\wtilde H\big )(r)\in Y$, when $r\in[0,T]$. Proposition \ref{K} gives continuity for the map $r\mapsto \wtilde H_r$. As $\Big(\wtilde H_r+\alpha\Big)\mapsto\Big(\wtilde H_r+\alpha\Big)^{-1}$ is continuous operation we see that $\boldsymbol{Z}_{\alpha} \big(\wtilde H\big )\in C([0,T],Y) $.
It remains to show that the norm estimate holds.
Having $\epsilon\in (0,1)$, $\norm{H-\wtilde H}_{C([0,T],Y)}\le\epsilon$, and $\alpha=2\epsilon^p$ we have 
\begin{align*}
 \norm{M_3-(\wtilde H+\alpha)}_{C([0,T],Y)}\le M_3-\frac{\alpha}{2}.
 \end{align*}
Thus $\Psi^Z_{\alpha}(\norm{M_3-(\wtilde H+\alpha)}_{C([0,T],Y)})=1$ and $\boldsymbol{Z}_{\alpha} \big(\wtilde H\big )$ is the map
\begin{align*}
r\mapsto\Big(\wtilde H_r+\alpha\Big)^{-1}.
\end{align*}
Let $r\in [0,T]$. We denote
\begin{equation*}
H_{\alpha,r}=(H_r+\alpha),\quad
\wtilde{H}_{\alpha,r}=(\wtilde{H_r}+\alpha),\quad
E=\wtilde{H}_{\alpha,r}-H_{\alpha,r}.
\end{equation*}
As $H_r$ is positive semidefinite we have
\begin{equation}
\label{kana}
\norm{H_{\alpha,r}^{-1}}_{Y}\le \alpha^{-1}.
\end{equation}
Moreover
\begin{equation*}
\wtilde{H}_{\alpha,r}^{-1}-H_{\alpha,r}^{-1}=
\big(\big[I+H_{\alpha,r}^{-1}E\big]^{-1}-I\big)H_{\alpha,r}^{-1}.
\end{equation*}
Thus 
\begin{equation}
\label{sikaniskat}
\norm{(\wtilde{H}_{\alpha,r})^{-1}-(H_{\alpha,r})^{-1}}_{Y}\le
\frac{\norm{(H_{\alpha,r})^{-1}E}_{Y}}{1-\norm{(H_{\alpha,r})^{-1}E}_{Y}} \norm{(H_{\alpha,r})^{-1}}_{Y}.
\end{equation}
We have $\frac{1}{2}\ge\frac{\epsilon}{\alpha}$.
Using (\ref{kana}) and (\ref{sikaniskat}) we have 
\begin{equation*}
\label{amatoestim}
\norm{(\wtilde{H}_{\alpha,r})^{-1}-(H_{\alpha,r})^{-1}}_{Y}\le
\frac{\frac{\epsilon}{\alpha}}{1-\frac{1}{2}} \norm{(A_{\alpha,r})^{-1}}_{Y}
\le 2\frac{\epsilon}{\alpha^2}=2^{-1}\epsilon^{1-2p}.
\qedhere
\end{equation*}
\end{proof} 
Let $P_r$ and $B$ be as defined in (\ref{operaattorit}). We define 
\begin{align}
\label{innerproduct}
& \boldsymbol{S}: C([0,T],Y) \to C([0,T]),
\\\nonumber  & \boldsymbol S\big(\wtilde Z_{\alpha}\big)(r) =\bra \wtilde Z_{\alpha}(r)P_r B1,B1\cet _{L^2(0,2T)}.
\end{align}
\begin{proposition}
\label{S}
We have $\norm{\boldsymbol{S}}_{C([0,T])}\le \frac{T^3}{3}$. 
\end{proposition}
\begin{proof}
As the maps $r\mapsto P_r B1$ and $r\mapsto Z_{\alpha}(r)$ are continuous, we have that $\boldsymbol S\big(\wtilde Z_{\alpha}) \in C([0,T])$.  
Let $r\in[0,T]$. We have 
\begin{align*}
 \norm{P_r}_{Y}\le 1,\quad \norm{B1}^2_{Y}= \frac{T^3}{3},
 \end{align*}
and therefore
\begin{align*}
&
|\boldsymbol S\big(\wtilde Z_{\alpha}\big)(r)| =|\bra \wtilde Z_{\alpha}(r)P_r B1,B1\cet _{L^2(0,2T)}|\le \frac{T^3}{3}\norm{\wtilde Z_{\alpha}(r)}_{Y}.
\qedhere
\end{align*} 
\end{proof} 
\begin{lemma}
\label{lemmah1}
Let $c\in\mathcal{D}(\mathcal A)$. There is $C>0$ such that for all $r>0$ and $p \in H^1(M)$ satisfying $supp(p)\subset M(r)$ there is $f\in S_r$ such that $u^{f}(x,T)=p(x)$ and
\begin{align}
\label{tauski}
\norm{f}_{L^2(0,2T)}\le C\norm{p}_{H^1(M)}.
\end{align}
\end{lemma}
We recall that $M(r)$ is defined by (\ref{Mr}) and $S_r$ is defined by (\ref{muumipappa}).
\begin{proof}
Let us consider the wave equation with time and space having the exchanged roles
\begin{align}
\label{eq_wave_boundary_sourset2}
&(\p_x^2-c(x)^{-2} \p_t^2)\wtilde u(x,t) = 0, & (x,t) \in (0,\chi(T))\times (0,T),
\\\nonumber& \wtilde u(x,T) = p(x), 
& x \in [0,\chi(T)],
\\& \nonumber\wtilde u(\chi(T),t) = \p_x\wtilde u(\chi(T),t) = 0,  & t \in (0,T).
\end{align}
By \cite{Lasiecka1986} the solution of (\ref{eq_wave_boundary_sourset2}) satisfies
\begin{align}
\norm{\wtilde{u}(0,\cdotp)}_{H^1(0,T)} \leqslant
C\norm{p}_{H^1(M(T))}.
\end{align}
If $supp(p)\subset M(r)$ then $supp(\wtilde{u}(0,\cdotp)) \subset [T-r,T]$ and $\wtilde u(x,0) = \p_t \wtilde u(x,0) = 0$, when $x\in [0,\chi(T)]$, by finite speed of propagation.
We choose $f(t)=\wtilde u(0,t)$.
\end{proof}
Let $f_{\alpha,r}$ be as in (\ref{minimizer}) and define 
\begin{align}
\label{toukollepatsas}
s_\alpha\in C([0,T]),\quad s_\alpha(r):=\bra f_{\alpha,r}, B1\cet _{L^2(0, 2T)}.
\end{align}
\begin{lemma}
\label{estimate-f}
Let $\alpha\in (0,\min(1,\frac{1}{\chi(T)^2}))$. Let $V$ be as defined in (\ref{Vee}). Then there is $C>0$, independent $\alpha$, such that 
\begin{align*}
\norm{s_\alpha-V}_{C([0,T])} \le C\alpha^{\frac{1}{4}}.
\end{align*}
\end{lemma}
\begin{proof}
Let $r\in [0,T]$ and $\delta>0$. Let us define $w_\delta\in H^1(M)$ 
 \begin{displaymath}
 w_\delta (x) = \left\{ \begin{array}{cl}
1, & \textrm{if\quad $x \in (0,\chi(r))$},\\
1-\frac{x-\chi(r)}{\delta}, & \textrm{if\quad $x \in [\chi(r),\chi(r)+\delta]$},\\
0, & \textrm{if\quad $x \in (\chi(r)+\delta,\infty)$}.\\
\end{array} \right.
\end{displaymath}
Using $c(x)>C_0$ we have
\begin{equation}
\label{ufepsilon}
\norm{w_{\delta}- 1_{M(r)}}_{L^2(M; dV)}^2 \le \frac{\epsilon}{3C_0^2}.
\end{equation}
When $\delta\in (0,\min(1,\frac{1}{\chi(T)}))$ we have 
\begin{equation}
\label{änkyräkänni}
\norm{w_\delta}_{H^1(M)}^2 
\le \chi(T)+\frac{\delta}{3} + \frac{1}{\delta}
\le \frac{3}{\delta}.
\end{equation}
Below $C>0$ denotes a constant that may grow between inequalities, and that depends only on $m,C_0,C_1,L$.
Lemma \ref{lemmah1} gives us $f_\delta$ for which $u^{f_\delta}(x,T)=w_\delta (x)$. Thus (\ref{änkyräkänni}) implies 
\begin{equation}
\label{normiestimpeelle}
\norm{f_\delta}_{L^2(0,2T)} 
\le C\norm{w_\delta}_{H^1(M)}\le \frac{C}{\delta^{\frac{1}{2}}}.
\end{equation}
Let $f\in S_r$. We define 
\begin{equation}
\label{Geealfa}
G_{\alpha,r}(f)= \norm{u^f - 1_{M(r)}}_{L^2(M; dV)}^2 + \alpha \norm{f}^2_{L^2(0,2T)}.
\end{equation}
Using (\ref{ufepsilon}) and (\ref{normiestimpeelle}) we have
\begin{equation}
\label{Gfeps}
G_{\alpha,r}(f_\delta)=\norm{w_\delta - 1_{M(r)}}_{L^2(M;dV)}^2
 +\alpha \norm{f_\delta}_{L^2(0,2T)}^2 \le \frac{\delta}{C}+\alpha\frac{C}{\delta}.
\end{equation}
Functional (\ref{Geealfa}) and the functional defined in Theorem \ref{thm_minimization} have the same minimizer $f_{\alpha,r}$. Using
(\ref{Blago_harmonic}), (\ref{possu}), and (\ref{toukollepatsas}) we have 
\begin{align*}
&\norm{s_\alpha-V}^2_{C([0,T])} = \sup_{r\in [0,T]}\big| \bra f_{\alpha,r},B1\cet _{L^2([0,2T])}- V(r) \big|^2.
\\\nonumber &\qquad \qquad\qquad\quad = \sup_{r\in [0,T]}\big| 
\bra u^{f_{\alpha,r}}(T),1\cet _{L^2(M; dV)} - \bra 1_{M(r)},1\cet _{L^2(M; dV)} \big|^2
\\\nonumber &\qquad \qquad\qquad  \le C\sup_{{r\in [0,T]}}\norm{u^{f_{\alpha,r}}(T)-1_{M(r)}}^2_{L^2(M; dV)}
\le C\sup_{{r\in [0,T]}} G_{\alpha,r}(f_{\alpha,r})
\\\nonumber &\qquad \qquad\qquad \le C\sup_{{r\in [0,T]}} G_{\alpha,r}(f_\delta)
\end{align*}
Using (\ref{Gfeps}) and choosing  $\delta=\alpha^{\frac{1}{2}}$ we have
\begin{align*}
&\norm{s_\alpha-V}^2_{C([0,T])} \le   C\alpha^{\frac{1}{2}}.
\qedhere
\end{align*}
\end{proof}
\begin{lemma}
\label{smoothv}
Let $V$ be as defined in (\ref{Vee}). Let $v$ be as defined in (\ref{pikkuvee}). If $c\in\mathcal{D}(\mathcal A)$ then there exists $\wtilde m>0$ s.t.
\begin{align*}
\norm{v}_{C^2([0,T])}\le \wtilde m\quad \hbox{and }
\norm{V}_{C^3([0,T])}\le \wtilde m. 
\end{align*}
\begin{proof}
Equations (\ref{kiipilkku}), (\ref{Vee}), (\ref{pikkuvee}), and (\ref{cee}) with the chain rule and the formula for the derivatives of inverse functions give us the result.
\end{proof}
\end{lemma}
For small $h>0$ we consider the partition
\begin{align*} 
(0,T)=(0,h) \cup[h,2h)\cup[2h,3h)\cup...\cup
[Nh-h,Nh)\cup [Nh,T),
\end{align*}
where $N\in \mathbb{N}$  satisfies $T-h\le Nh < T$.
We define a discretized and regularized approximation of the derivative operator $\p_r$ by
\begin{align} 
\label{differenssi}
D_{h}:C([0,T])\to L^\infty(0,T), 
\end{align}
\begin{displaymath}
D_{h}(\wtilde s_{\alpha})(r)= \left\{ \begin{array}{cl}
\frac{\wtilde s_{\alpha}(h)}{h}, & \textrm{if\quad $r \in (0,h)$},\\
\frac{\wtilde s_{\alpha}(jh+h
)-\wtilde s_{\alpha}(jh)}{h}, & \textrm{if\quad $r \in [jh,jh+h)$},\\
 \frac{\wtilde s_{\alpha}(T)-\wtilde s_{\alpha}(Nh)}{h},& \textrm{if\quad $r \in [Nh,T)$.}\\ 
\end{array} \right.
\end{displaymath}
\begin{proposition}
\label{Z4}
Let $\beta>0$ and $\epsilon\in (0,\min(\frac{1}{\beta^{\frac{1}{4}}},\frac{1}{\beta^{\frac{1}{4}} \chi(T)^{\frac{1}{2}}}))$. Let $\alpha = \beta\epsilon^4$, $h=\epsilon^{\frac{1}{2}}$, $V$ be as defined in (\ref{Vee}) and let $s_\alpha$ be as defined in (\ref{toukollepatsas}). Let us assume that $\wtilde{s}_\alpha\in  \mathcal{B}_{C([0,T])} (s_\alpha,\epsilon)$. Then
\begin{align*}
\norm{D_{h}(\wtilde s_{\alpha})-\p_r V}_{L^\infty (0,T)}
\le C\epsilon^{\frac{1}{2}},
\end{align*} 
where C is independent of $\alpha$ and $\wtilde s_{\alpha}$.
\end{proposition}
\begin{proof}
Let $r \in [jh,jh+h)$. Using (\ref{differenssi}) we have
\begin{align*}
&\Big|D_{h}(\wtilde{s}_{\alpha})(r)-\p_r V(r)\Big|=
\Big|\frac{\wtilde{s}_\alpha(jh
+h)-\wtilde{s}_\alpha(jh)}{h}
-\p_r V(r)\Big|
\\\nonumber& \le\Big|\frac{\wtilde{s}_\alpha(jh
+h)-s_\alpha(jh
+h)}{h}\Big|
+\Big|\frac{s_\alpha(jh)-\wtilde{s}_\alpha(jh)}{h} \Big|
\\\nonumber& +\Big|\frac{s_\alpha(jh
+h)-V(jh+h)}{h}\Big|
+\Big|\frac{V(jh)-s_\alpha(jh)}{h} \Big|
\\\nonumber&
+\Big|\frac{V(jh
+h)-V(jh)}{h}-\p_r V(r)
 \Big|.
\end{align*}
Lemma \ref{smoothv} gives us $\norm{V}_{C^3([0,T])} \le \wtilde m$. When $r \in [jh,(jh+h)$ there is $\xi \in (jh,jh+h)$ such that 
\begin{align}
\label{apubapu}
\Big | \frac{V(jh+h)- V(jh)}{h}-\p_r V(r)\Big | =\Big | \p_r V(\xi)-\p_r V(r)\Big | \le h \wtilde m.
\end{align}
Using (\ref{apubapu}) and Lemma \ref{estimate-f} with assumption we get
\begin{align*}
\Big|D_{h}(\wtilde{s}_{\alpha})(r)-\p_r V(r)\Big| \le \frac{2\epsilon}{h}+\frac{2 C\alpha^{\frac{1}{4}}}{h}+h\wtilde m.
\end{align*}
Let us choose
$h=\epsilon^{\frac{1}{2}}$ and $\alpha = \beta \epsilon^4$. Then 
\begin{align}
\label{rusakko1}
\Big|D_{h}(\wtilde{s}_{\alpha})(r)-\p_r V(r)\Big| \le  C\epsilon^{\frac{1}{2}}.
\end{align}
The proof is almost identical when $r\in(0,h)$ or $r\in[Nh,T)$. Note that the right hand side of (\ref{rusakko1}) is independent of $r$. \\
\end{proof}
Let $C_0$ and $C_1$ be as in (\ref{nopeudet}). Let $\wtilde k_\alpha\in L^\infty(0,T)$ and we define
\begin{equation*}
\Psi^W(\wtilde k_\alpha)(r)= \left\{ \begin{array}{ll}
\frac{1}{C_1}, & \textrm{if $\quad \wtilde k_\alpha(r) < C_1^{-1}$},\\
\frac{1}{\wtilde k_\alpha(r)}, & \textrm{if   $\quad C_1^{-1}\le \wtilde k_\alpha(r)\le C_0^{-1}$},\\
 \frac{1}{C_0}, & \textrm{if $\quad \wtilde k_\alpha(r) > C_0^{-1}$}.\\
\end{array} \right.
\end{equation*}
We define 
\begin{equation}
\label{simosalminen}
W: L^\infty(0,T)\to L^\infty(M),\quad W(\wtilde k_\alpha)(r)= \left\{ \begin{array}{cl}
\Psi^W(\wtilde k_\alpha)(r), & \textrm{if $r \in (0,T)$},\\
1, & \textrm{if $r \in [T,\infty)$}.\\
\end{array} \right.
\end{equation}
\begin{proposition}
\label{Z5}
Let $V$ be as defined in (\ref{Vee}) and $v$ be as defined in (\ref{pikkuvee}). Let us assume that $\wtilde{k}_{\alpha}\in \mathcal{B}_{L^\infty (0,T)}(\p_r V,\epsilon)$. Then 
\begin{align*}
\norm{ W(\wtilde{k}_{\alpha})-v}_{L^\infty (M)} 
\le C_1^2\epsilon.
\end{align*}
\end{proposition} 
\begin{proof}
For all $x\in M,$ we have $ 0<C_0\le c(x)\le C_1$. Let $r\in (0,T)$ and let assume that $C_1^{-1}\le \wtilde k_\alpha(r)\le C_1^{-1}$. Using (\ref{pikkuvee}) and (\ref{cee}) we have $0<\frac{1}{C_1}\le\p_r V(r)\le\frac{1}{C_0}$. Then 
\begin{align}
\label{saksalanteemu}
&\Big|{\frac{1}{\wtilde{k}_{\alpha}(r)}-\frac{1}{\p_r V(r)}}\Big |=\Big| \frac{\wtilde{k}_{\alpha}(r)-\p_r V(r)}{\wtilde{k}_{\alpha}(r)\p_r V(r)}\Big|
\le C_1^2\epsilon.
\end{align}
In the case when $r\in (0,T)$ and $\wtilde k_\alpha(r)<C_1^{-1}$ or $\wtilde k_\alpha(r)> C_0^{-1}$ we obtain similar estimates. Note that the right hand side of (\ref{saksalanteemu}) is independent of $r$. When $r\ge T$ the left hand side is identically zero.
\qedhere
\end{proof} 
Let $\wtilde w_\alpha\in L^\infty(M)$ and we define 
\begin{equation}
\label{hilbertinukko}
\Psi^\Phi:L^\infty(M)\to L^\infty(M),\quad \Psi^\Phi(\wtilde w_\alpha)(r):= \left\{ \begin{array}{ll}
C_0, & \textrm{if $\quad \wtilde w_\alpha(r) < C_0$},\\
\wtilde w_\alpha(r), & \textrm{if   $\quad C_0\le \wtilde w_\alpha(r)\le C_1$},\\
C_1, & \textrm{if $\quad w_\alpha(r) > C_1$}.\\
\end{array} \right.
\end{equation}
Let $\wtilde w_\alpha\in L^\infty(M)$ and we define 
\begin{align} 
\label{perttioinonen}
\Upsilon:L^\infty(M)\to C(M),\quad \Upsilon(\wtilde w_\alpha)(t) = \int_0^t\wtilde w_\alpha(t^\prime) dt^\prime.
\end{align}
Using (\ref{hilbertinukko}) and (\ref{perttioinonen}) we see that $\Upsilon\circ\Psi^\Phi(\wtilde w_\alpha):M\to M$ is bijective as a function of $t$. Let us denote $\wtilde \chi=\Upsilon\circ\Psi^\Phi (\wtilde w_\alpha)$ and $\wtilde \chi^{-1}=(\Upsilon\circ\Psi^\Phi(\wtilde w_\alpha))^{-1}$.
We define the sixth operator by
\begin{equation}
\label{opePFI}
 \Phi: L^\infty(M)\to L^\infty(M),
 \quad\Phi(\wtilde w_\alpha)= \left\{ \begin{array}{cl}
\wtilde w_\alpha\circ\wtilde\chi^{-1}, & \textrm{if\quad $x \in [0,L)$},\\
1, & \textrm{if\quad $x \in [L,\infty)$}.\\
\end{array} \right.
\end{equation}
\begin{proposition} 
\label{Z6}
Let $\epsilon>0$ and $v$ be as defined in (\ref{pikkuvee}). Let us assume that $\wtilde{w}_{\alpha}\in \mathcal{B}_{L^\infty (M)}(v,\epsilon)$. Then 
\begin{align*}
\norm{\Phi (\wtilde w_\alpha)-c}_{L^\infty (M)} 
\le C\epsilon.
\end{align*}
\end{proposition}
\begin{proof}
Let us denote $t=\chi^{-1}(x)$ and $\wtilde t=\wtilde\chi^{-1}(x)$. Let $x\in [0,L)$. Using (\ref{cee}) and (\ref{opePFI}) we have
\begin{align*}
&| \Phi (\wtilde w_\alpha)(x)-c(x)|=| \wtilde w_\alpha(\wtilde t)-v(t)|\le 
|\wtilde w_\alpha(\wtilde t)-v(\wtilde t)|+| v(\wtilde t)-v(t)|.
\end{align*}
Lemma \ref{smoothv} gives us $\norm{v}_{C^2(0,T)}\le \wtilde m$ and we have 
\begin{align}
\label{kisukisu}
| v(\wtilde t)-v(t)|
\le\wtilde m |\wtilde t- t|.
\end{align}
Using (\ref{nopeudet}) and (\ref{cee}) we see that $0<C_0\le v(t)\le C_1$ and we have
\begin{align}
\label{estimal}
C_0|\wtilde t- t|\le|\int_t^{\wtilde t} v(t^\prime) dt^\prime|= |\chi(\wtilde t)- \chi(t)| .
\end{align}
Having $\wtilde \chi(\wtilde t)=x=\chi(t)$ and using (\ref{cee}) and (\ref{perttioinonen}) we see that
\begin{align}
\label{estimyl}
|\chi(\wtilde t)- \chi(t)|=|\chi(\wtilde t)- \wtilde \chi(\wtilde t)|
=|\int_0^{\wtilde t} (v(t^\prime)-\wtilde{w}_{\alpha}(t^\prime)) dt^\prime|\le \wtilde \chi^{-1}(L) \epsilon.
\end{align}
Using (\ref{kisukisu}),(\ref{estimal}), and (\ref{estimyl}) we have
\begin{align}
\label{koriolutta}
&| \Phi (\wtilde w_\alpha)(x)-c(x)|\le\Big(1+\frac{\wtilde m\wtilde \chi^{-1}(L)}{C_0}\Big )\epsilon.
\end{align}
Note that the right hand side in (\ref{koriolutta}) does not depend on $x$. When $x\in[T,\infty)$ the left hand side in identically zero.
\end{proof} 
\begin{proof}[Proof of Theorem \ref{kaiken_teoria}.] Let
\begin{align} 
\label{epsilonnolla}
\epsilon_0=min\{1,\frac{1}{2T},\frac{1}{2^{\frac{13}{4}}T^{9} },\frac{1}{2^{\frac{13}{4}}T^{9}\chi(T)^{\frac{9}{2}} }\}.
\end{align}
 Suppose that $\wtilde\Lambda\in\mathcal{B}_{Y}(\Lambda,\epsilon) $
We denote $H=\boldsymbol{H}\Lambda$ and $\wtilde H=\boldsymbol{H}\wtilde\Lambda$. Using Proposition \ref{K} we get 
\begin{align*}
\norm{H-\wtilde{H}}_{{C([0,T],Y)}}\le2T\epsilon.
\end{align*}
We denote $Z_{\alpha}=\boldsymbol{Z}_{\alpha} \big( H\big )$ and $\wtilde Z_{\alpha}=\boldsymbol{Z}_{\alpha} \big(\wtilde H\big ) $. We have $\wtilde H\in \mathcal{B}_{{C([0,T],Y)}}(H,2T\epsilon)$ and $\epsilon\in \big (0,\min(1,\frac{1}{2T})\big )$. Proposition \ref{Z2} with $p=\frac{4}{9}$ gives us 
\begin{align*}
\norm{Z_{\alpha} - \wtilde Z_{\alpha}}_{C([0,T],Y)} \le
2^{-2p} T^{1-2p}\epsilon^{1-2p}=:\epsilon_1,
\end{align*}
since $\alpha=2^{p+1}T^p\epsilon^p=2^{\frac{13}{9}}T^{\frac{4}{9}}\epsilon^{\frac{4}{9}}$.\\
We denote $s_{\alpha} =\boldsymbol{S}Z_{\alpha}$ and $\wtilde s_{\alpha} =\boldsymbol{S}\wtilde Z_{\alpha}$. We have $\wtilde Z_{\alpha}\in \mathcal{B}_{C([0,T],Y)}(Z_{\alpha},\epsilon_1)$. Proposition \ref{S} gives us
\begin{align*}
\norm{s_{\alpha} - \wtilde s_{\alpha}}_{C([0,T])} \le
 \frac{2^{-2p}}{3}T^{4-2p}\epsilon^{1-2p}=:\epsilon_2.
\end{align*}
We denote $\wtilde k_{\alpha}=D_{h}(\wtilde s_{\alpha})$. We have $\wtilde s_{\alpha}\in \mathcal{B}_{C([0,T])}(s_{\alpha},\epsilon_2)$ and $\epsilon_2\in \Big(0,\min(\frac{1}{\beta^{\frac{1}{4}}},\frac{1}{\beta^{\frac{1}{4}} \chi(T)^{\frac{1}{2}}})\Big )$. Proposition \ref{Z4} with $\beta=3^42^5T^{-\frac{76}{9}}$ gives us 
\begin{align*}
\norm{\wtilde k_{\alpha}-\p_r V}_{L^\infty (0,T)}
\le C \frac{2^{-p}}{3^{\frac{1}{2}}}T^{2-p}\epsilon^{\frac{1}{2}-p}=:\epsilon_3,
\end{align*} 
since $\alpha = \beta\frac{2^{-8p}}{3^4}T^{16-8p}\epsilon^{4-8p}=2^{\frac{13}{9}}T^{\frac{4}{9}}\epsilon^{\frac{4}{9}}$.  \\
We denote $\wtilde{w}_{\alpha}=W(\wtilde{k}_{\alpha})$. We have $\wtilde{k}_{\alpha}\in \mathcal{B}_{L^\infty (0,T)}(\p_r V,\epsilon_3)$. Proposition \ref{Z5} gives us 
\begin{align*}
\norm{ \wtilde{w}_{\alpha}-v}_{L^\infty (M)} 
\le C_1^2C \frac{2^{-p}}{3^{\frac{1}{2}}}T^{2-p}\epsilon^{\frac{1}{2}-p}=:\epsilon_4.
\end{align*}
We denote $\wtilde{c}_{\alpha}=\Phi (\wtilde w_\alpha)$. We have $\wtilde{w}_{\alpha}\in \mathcal{B}_{L^\infty (M)}(v,\epsilon_4)$. Proposition \ref{Z6} gives us 
\begin{align*}
\norm{ \wtilde{c}_{\alpha}-c}_{L^\infty (M)} 
\le C C_1^2 \frac{2^{-p}}{3^{\frac{1}{2}}}T^{2-p}\epsilon^{\frac{1}{2}-p}.
\end{align*}
Let $\epsilon \in (0,\epsilon_0)$. 
Using (\ref{Hoo}),(\ref{Zalfa}),(\ref{innerproduct}),(\ref{differenssi}),(\ref{simosalminen}), and (\ref{opePFI}) we define 
\begin{align}
\label{Rstategia}
&\mathcal R_{\alpha(\epsilon)}:Y\to X,
\\\nonumber&\mathcal R_{\alpha(\epsilon)}=\Phi\circ W\circ D_{h}\circ\boldsymbol{S}\circ \boldsymbol{Z}_{\alpha}\circ\boldsymbol{H},
\end{align}
and we have an estimate
\begin{align*}
\norm{\mathcal R_{\alpha(\epsilon)}(\wtilde{\Lambda})-c}_X \le
C\epsilon^{\frac{1}{18}}.
\end{align*}
\end{proof}

\section*{Appendix A: The direct problem}
\label{AAA}

For  the convenience of reader, we give the proof of the following, quite well known result, for continuity of the
direct problem in our setting.

\begin{theorem}
\label{tomoffinland}
Let $c\in\mathcal{D}(\mathcal A)$ and $f\in L^2(0,2T)$. Then the boundary value problem (\ref{dartwader}) has a unique solution $u^f\in H^1((0,2T) \times M)$. The operators $\Lambda$ and $U_{T}$, defined in (\ref{lam}) and (\ref{nasu}), are linear and bounded operators, and the direct map $\mathcal A$, defined in (\ref{apina}), is continuous. 
\end{theorem}
\begin{proof}Let consider the wave equation (\ref{dartwader}).
When $c=1$ on $M$ we denote the solution by $u_0^f$ and have
$$
u_0^f(t,x) = h(t-x), \quad 
h(s) = \begin{cases}
-\int_0^s f(t) dt, & t > 0,
\\
\qquad 0, & t\le 0.
\end{cases} 
$$
Notice that $f \mapsto u_0^f$ is continuous from $L^2(0,2T)$ to $H^1((0,2T) \times M)$.
Let us now show that the same is true for $f \mapsto u^f$.

We choose $\psi \in C^\infty(M)$ such that $\psi = 1$ near $x=0$
and $c=1$ in the support of $\psi$.
The commutator $A = [\p_x^2, \psi]$ is a first order differential operator,
whence $A u_0^f \in L^2((0,2T) \times M)$ for $f \in L^2(0,2T)$.
Let $w$ be the solution of 
\begin{align*}
&(\p_t^2 - c(x)^2 \p_x^2) w(t,x) = A u_0^f \quad \text{in $(0,2T) \times M$},
\\\nonumber
&\p_x w(t,0) = 0,
\\\nonumber
&w|_{t = 0} = \p_t w|_{t=0} = 0.
\end{align*}
Then $w \in H^1((0,2T) \times M)$, see e.g. \cite {Ladyzhenskaya1985}, and $u^f := \psi u_0^f - w \in H^1((0,2T) \times M)$ is the solution of (\ref{dartwader}).
Indeed, $c^2 \p_x^2 \psi = \psi \p_x^2 + A$, 
where $\psi$ is interpreted as a multiplication operator.
Here we are using the fact that $c=1$ in the support of $\psi$.
Thus
$$
(\p_t^2 - c^2 \p_x^2) u^f
= \psi (\p_t^2 - \p_x^2) u_0^f + A u_0^f - A u_0^f = 0.
$$
As $\psi = 1$ near $x=0$, we see that $u$ satisfies also the boundary conditions in (\ref{dartwader}).

Let us now suppose that $f\in C_0^{\infty} (0,T)$. Let $u^f$ be solution for the boundary value problem in (\ref{dartwader}). Let $c$ be as defined in (\ref{nopeudet}). Let $x\in M$ and we define 
\begin{align}
\label{kanki}
k\in C^2(M),\quad k(x)=exp(\frac{1}{2c(x)})
\end{align}
and 
\begin{align}
\label{bongi}
& G:C^2(M\times (0,2T))\to C(M\times (0,2T)),
\\\nonumber& G(u)=k^{-1}\Big(\p_t^2-c^2\p_x^2\Big)ku=
\Big(\p_t^2-c^2\p_x^2-2c^2k^{-1}\p_xk\p_x-c^2k^{-1}\p_x^2k\Big)u.  
\end{align}
Let $x\in M$ and define  
\begin{align}
\label{bongi2}
\phi (x)= \int_{0}^{x}c({x^\prime})^{-1}d{x^\prime}. 
\end{align}
Let us denote $\wtilde x=\phi (x)$ and define
\begin{align}
\label{kanki2}
\wtilde u^f(\wtilde x,t)=\wtilde u^f(\phi(x),t):=\frac{u^f(x,t)}{k(x)}.
\end{align}
Using (\ref{kanki}), (\ref{bongi}),  (\ref{bongi2}), and (\ref{kanki2}) we see that $\wtilde u^f(\wtilde x,t)$ is a solution of the boundary value problem 
\begin{align}
\label{potentiaal}
&(\p_t^2-\p_{\wtilde x}^2+ q(\wtilde x))\wtilde u^f(\wtilde x,t) = 0, & (\wtilde x,t) \in M\times (0,2T),
\\\nonumber& \p_{\wtilde x} \wtilde u^f(0,t) = e^{\frac{1}{2}}f(t), 
\quad \p_x \wtilde u^f(L,t) = 0, & t \in (0, 2T),
\\\nonumber& \wtilde u^f(x,0) = \p_t \wtilde u^f(x,0) = 0,  & \wtilde x \in M,
\end{align}
where
\begin{align}
\label{kuuukko}
q(\wtilde{x})=-c^2(\phi^{-1}(\wtilde{x}))k^{-1}(\phi^{-1}(\wtilde{x}))\p_x^2k(\phi^{-1}(\wtilde{x})).
\end{align} 
We define $\lambda_qf=\wtilde u|_{\wtilde x=0}$. Let us consider two velocity functions $c_1$ and $c_2$, and let $q_1$ and $q_2$ be the potentials corresponding to $c_1$ and $c_2$ via formula (\ref{kuuukko}) . Using (\ref{kanki2}) we have
\begin{align}
\label{akuankka1}
\norm{\Lambda_{c_1}-\Lambda_{c_2}}_{\mathcal{L} (L^2(0,2T))} \le e^{\frac{1}{2}}\norm{\Lambda_{q_1}-\Lambda_{q_2}}_{\mathcal{L} (L^2(0,2T))}. 
\end{align}
Let us denote by $u^f_{q_1}$ and $u^f_{q_2}$ the two solutions with respect to potentials $q_1$ and $q_2$ for the problem (\ref{potentiaal}). Let us define  
$w(\wtilde x,t)=\wtilde u^f_{q_1}(\wtilde x,t)-\wtilde u^f_{q_2}(\wtilde x,t)$. Then $w$ is the solution of 
\begin{align}
\label{herkulesbaari}
&(\p_t^2-\p_x^2+ q_1(\wtilde x))w(\wtilde x,t) = F(\wtilde x,t), & (\wtilde x,t) \in M\times (0,2T),
\\\nonumber& \p_{\wtilde x} w(0,t) = 0, 
\quad \p_{\wtilde x} w(L,t) = 0, & t \in (0, 2T),
\\\nonumber& w(\wtilde x,0) = \p_t w(\wtilde x,0) = 0,  & x \in M,
\end{align}
where 
\begin{align}
\label{isosika}
F(\wtilde x,t)=(q_1(\wtilde x)-q_2(\wtilde x)) \wtilde u^f_{q_2}(\wtilde x,t).
\end{align}
By \cite{Lasiecka1986} we have an estimate 
 \begin{align}
 \label{kontulankingkong}
\norm{w}_{H^1(M\times (0,2T))} \le C\norm{F}_{L^2 (M\times (0,2T))}.
\end{align}
Let $\hat T= \phi (L)$ and  let $supp (u^f_{q_2})\subset [0,\hat T]$. Using (\ref{isosika}) we have
 \begin{align}
  \label{kontulankingkong2}
\norm{F}_{L^2 (M\times (0,2T))} \le \norm{q_1-q_2}_{L^{\infty}(0,\hat T)} \norm{u^f_{q_2}}_{H^1(M\times (0,2T))}.
\end{align}
Because $f \mapsto \wtilde u^f_{q_2}$ is continuous from $L^2(0,2T)$ to $H^1(M\times (0,2T))$ we have
\begin{align}
\label{röhröh}
\norm{\wtilde u^f_{q_2}}_{H^1(M\times (0,2T))}
 \le C\norm{f}_{L^2(0,2T)}.
\end{align}
Using (\ref{lam}) we see that
\begin{align}
\label{akuankka2}
\norm{\Lambda_{q_1}f-\Lambda_{q_2}f}_{L^2(0,2T)} 
\le C\norm{u^f_{q_1}-u^f_{q_2}}_{H^1(M\times (0,2T))}. 
\end{align} 
Using (\ref{kontulankingkong2}), (\ref{kontulankingkong}), (\ref{röhröh}), and (\ref{akuankka2}) we have
\begin{align}
\label{akuankka3}
\norm{\Lambda_{q_1}-\Lambda_{q_2}}_{\mathcal{L} (L^2(0,2T))} 
\le C\norm{q_1-q_2}_{L^{\infty}(0,T)}. 
\end{align} 
Using (\ref{kuuukko}) we have 
$q(\wtilde{x})=-c^2(x)k^{-1}(x)\p_x^2k(x)|_{x=\phi^{-1}(\wtilde{x})}$.
Let $x\in M$ and define 
\begin{align*}
 h_1(x)=-c_1^2(x)k_1^{-1}(x)\p_x^2k_1(x),\quad h_2(x)=-c_2^2(x)k_2^{-1}(x)\p_x^2k_2(x).
\end{align*} We have
\begin{align*}
& |h_1(x)-h_2(x)|=|c_2^2(x)k_2^{-1}(x)\p_x^2k_2(x)-c_1^2(x)k_1^{-1}(x)\p_x^2k_1(x)|
\\\nonumber&\le|c_1^2(x)-c_2^2(x)||k_1^{-1}(x)||\p_x^2k_1(x)|
\\\nonumber&+|k_1^{-1}(x)-k_2^{-1}(x)||c_2^2(x)||\p_x^2k_1(x)|
\\\nonumber&+|\p_x^2k_1(x)-\p_x^2k_2(x)||c_2^2(x)||k_2^{-1}(x)|
\end{align*}
Using (\ref{kanki}) we can bound each of these tree terms and have 
 \begin{align}
 \label{akuankka3}
\norm{q_1-q_2}_{L^{\infty}(0,\hat T)}\le C\norm{c_1-c_2}_{C^2(0,L)}.
\end{align} 
Now (\ref{akuankka1}), (\ref{akuankka2}), and (\ref{akuankka3}) imply that $\mathcal A$ is continuous.
\end{proof}

\noindent{\bf Acknowledgements.} We thank Samuli Siltanen for inspiring discussions on the regularization on inverse problems. 

L.\ Oksanen was partly supported by 
EPSRC.
J.\ Korpela and M.\ Lassas were  supported by Academy of Finland, grants 273979 and 284715.

\bibliographystyle{abbrv} 
\bibliography{main}
\end{document}